\pdfoutput=1
\documentclass[10pt, twoside]{article}

\usepackage{amssymb,epsfig,amsmath,enumitem,eufrak, mathrsfs}
\usepackage{amsthm}
\usepackage{mathtools}
\usepackage{dsfont}
\usepackage{amsfonts}
\usepackage{indentfirst}
\usepackage[square,numbers]{natbib}
\usepackage{graphicx}
\usepackage{caption}
\usepackage{subcaption}
\usepackage{eso-pic}
\usepackage{chngcntr}
\usepackage[top=1.25in, bottom=1.25in, left=1in, right=1in]{geometry}
\usepackage[utf8]{inputenc}
\usepackage{datetime}
\usepackage{apptools}
\usepackage{titlesec}
\definecolor{darkred}{rgb}{0.5,0.2,0.2}
\usepackage[pdftitle={On the weak-hash metric for boundedly finite integer-valued measures},pdfauthor={Maxime Morariu-Patrichi},colorlinks=TRUE,allcolors=darkred]{hyperref}
\AtAppendix{\counterwithin{marked_point_process}{subsection}}

\newdateformat{mydate}{\THEDAY\ \monthname[\THEMONTH] \THEYEAR}

\titleformat{\subsubsection}[runin]
  {\normalfont\normalsize\bfseries}{\thesubsubsection}{1em}{}

\makeatletter
\def\@endtheorem{\endtrivlist}% NEW
\makeatother

\graphicspath{{plots/}}

\usepackage{fancyhdr}
\fancypagestyle{firstpg}{%
	\fancyhf{}
	\fancyfoot[C]{\small\thepage}
	\fancyhead[CO]{}
	\fancyhead[CE]{}
	
}

\newcommand{\borelSetsOf}[1]{\mathcal{B}(#1)} % Borel sets of a given set
\newcommand{\realLine}{\mathbb{R}}  % real line
\newcommand{\realLineNonNegative}{\mathbb{R}_{\geq 0}} % non-negative real numbers
\newcommand{\realLinePositive}{\mathbb{R}_{>0}} % positive real line
\newcommand{\integers}{\mathbb{N}} % integers
\newcommand{\anyCSMS}{\mathcal{X}}  % generic notation for any complete separable metric space
\newcommand{\anyCSMSBorel}{\borelSetsOf{\anyCSMS}}  % Borel sets of any complete separable metric space
\newcommand{\boundedlyFiniteMeasures}[1]{\mathcal{N}^{\#}_{#1}} % boundedly finite integer-valued counting measures
\newcommand{\boundedlyFiniteMeasuresBorel}[1]{\mathcal{B}(\boundedlyFiniteMeasures{#1})} % Borel sets of the space of boundedly finite integer-valued counting measures
\newcommand{\weakHash}{w^{\#}} % weak-hash topology
\newcommand{\weakHashDistance}[2]{d^{\#}(#1,#2)} % weak-hash distance
\newcommand{\weakHashDistancePlain}{d^{\#}} % weak-hash distance (symbol only)
\newcommand{\prohorovDistance}[2]{d(#1,#2)} % Prohorov distance
\newcommand{\prohorovDistancePlain}{d} % Prohorov distance
\newcommand{\collection}[2]{(#1)_{#2}} % collection of variables
\begin{document}

\thispagestyle{firstpg}
\begin{center}
\large \textbf{ON THE WEAK-HASH METRIC \\ FOR BOUNDEDLY FINITE INTEGER-VALUED MEASURES} \\
	\vspace{1cm}
	\large  Maxime Morariu-Patrichi\footnote{Department of Mathematics, Imperial College London, South Kensington Campus, London SW7 2AZ, UK. E-mail: \href{mailto:m.morariu-patrichi14@imperial.ac.uk}{\texttt{m.morariu-patrichi14@imperial.ac.uk}} URL: \url{http://www.maximemorariu.com}}  \\
	\vspace{0.5cm}
	\normalsize {\mydate\today}
	\vspace{1cm}
\end{center}

\begin{abstract}
	It is known that the space of boundedly finite integer-valued measures on a complete separable metric space becomes itself a complete separable metric space when endowed with the weak-hash metric. It is also known that convergence under this topology can be characterised in a way that is similar to the weak convergence of totally finite measures. However, the original proofs of these two fundamental results assume that a certain term is monotonic, which is not the case as we give a counterexample. We manage to clarify these original proofs by addressing specifically the parts that rely on this assumption and finding alternative arguments.
\end{abstract}

\bigskip

\bigskip

\numberwithin{equation}{section}

%%%%%%%%%%%%%%%%%%%%%%%%%%%%%%%%%%%%%%%%%%%%%%%%%%%%%
\section{Introduction} \label{subsec:metric_properties_of_measure_space}

Let $\anyCSMS$ be a complete separable metric space and $x_{0}\in\anyCSMS$ be a fixed origin. We denote by $B_{r}(x)$ the open ball with radius $r\in\realLineNonNegative$ and centre $x\in\anyCSMS$. We use the short notation $B_r := B_r(x_0)$ for the open balls centred at $x_0$. For any subset $A\subset \anyCSMS$ and $\varepsilon\in\realLinePositive$, the $\varepsilon$-neighbourhood of $A$ is defined by $A^\varepsilon:=\bigcup_{a\in A} B_\varepsilon(a)$, the boundary of $A$ is denoted by $\partial A$ and the closure of $A$ is denoted by $\overline{A}$. For any Borel measure $\xi$ on $\anyCSMS$ and any $r\in\realLineNonNegative$, we use the notation $\xi^{(r)}$ to refer to the restriction of $\xi$ to the open ball $B_{r}$, that is $\xi^{(r)}(A) = \xi(A\cap B_{r})$ for all $A\in\anyCSMSBorel$. A Borel measure $\xi$ on $\anyCSMS$ is called totally finite if $\xi(\anyCSMS)<\infty$. We denote by $\mathcal{M}_\anyCSMS$ the space of totally finite measures on $\anyCSMS$ and by $\prohorovDistancePlain$ the Prohorov distance on $\mathcal{M}_\anyCSMS$ defined by
\begin{alignat*}{2}
	d\,: \mathcal{M}_\anyCSMS\times\mathcal{M}_\anyCSMS &\rightarrow\realLineNonNegative && \\
	(\mu,\nu)& \mapsto d(\mu,\nu):=\inf \{\varepsilon\in\realLineNonNegative \,:\, && \mu(A) \leq \nu(A^\varepsilon)+\varepsilon \mbox{ and } \nu(A) \leq \mu(A^\varepsilon)+\varepsilon, \\
	& &&\mbox{for all closed } A\subset\anyCSMS \}.
\end{alignat*}
 It is known that $\prohorovDistancePlain$ makes $\mathcal{M}_\anyCSMS$ a complete separable metric space, see for example Section A2.5 in \citet[p.~398--402]{daleyVereJonesVolume1}.

In this paper, we are interested in boundedly finite integer-valued measures. A Borel measure $\xi$ on $\anyCSMS$ is called boundedly finite if $\xi(A)<\infty$ for all bounded Borel sets $A\in\borelSetsOf{\anyCSMS}$. We denote by $\boundedlyFiniteMeasures{\anyCSMS}$ the space of boundedly finite measures on $\anyCSMS$ with values in $\integers\cup\{\infty\}$. 
Note that such measures are always atomic (i.e., a superposition of Dirac measures), see for example Proposition 9.1.III.(ii) in \citet[p.~4]{daleyVereJonesVolume2}.
One might ask if the Prohorov distance $\prohorovDistancePlain$ on the space $\mathcal{M}_\anyCSMS$ has a counterpart on the space $\boundedlyFiniteMeasures{\anyCSMS}$. \citet[p.~403]{daleyVereJonesVolume1} tackle this question by considering the distance function
\begin{align} \label{eq:weak_hash_distance_definition}
	d^{\#}\,: \boundedlyFiniteMeasures{\anyCSMS}\times\boundedlyFiniteMeasures{\anyCSMS}&\rightarrow\realLineNonNegative \nonumber \\
	(\mu,\nu)&\mapsto\weakHashDistance{\mu}{\nu}:=\int_{0}^{\infty}e^{-r}\frac{d(\mu^{(r)},\nu^{(r)})}{1+d(\mu^{(r)},\nu^{(r)})}dr.
\end{align}
The core idea is to use the Prohorov metric on the restrictions to the open balls and compute a weighted average. They name the corresponding topology the $\weakHash$-topology (``weak-hash'') and refer to $\weakHashDistancePlain$ as the $\weakHash$-distance. They then obtain the following two fundamental results. The first one is a characterisation of convergence under this metric.
% Theorem: weak hash convergence implies weak convergence on spheres
\newtheorem{weakHash_implies_weak_on_spheres2}[]{Theorem}[section]
\begin{weakHash_implies_weak_on_spheres2}[Characterisation of convergence] \label{thm:weakHash_implies_weak_on_spheres2}
Let $\collection{\mu_{k}}{k\in\integers}$ be a sequence in $\boundedlyFiniteMeasures{\anyCSMS}$ and $\mu\in\boundedlyFiniteMeasures{\anyCSMS}$. Then, the following statements are equivalent:
\begin{enumerate}[label=\textup{(\roman*)}]
	\item 	$\weakHashDistance{\mu_{k}}{\mu}\rightarrow 0$ as $k\rightarrow\infty$;
	\item $\int_\anyCSMS f(x)\mu_k(dx)\rightarrow \int_\anyCSMS f(x)\mu(dx)$ as $k\rightarrow\infty$ for all bounded continuous functions $f$ on $\anyCSMS$ vanishing outside a bounded set;
	\item there exists an increasing sequence $\collection{r_{n}}{n\in\integers}$ with $r_{n}\rightarrow\infty$ as $n\rightarrow\infty$ such that \\
	$\prohorovDistance{\mu^{(r_{n})}_{k}}{\mu^{(r_{n})}}\rightarrow 0$ as $k\rightarrow\infty$ for all $n\in\integers$;
	\item $\mu_k(A)\rightarrow\mu(A)$ as $k\rightarrow\infty$ for all bounded sets $A\in\anyCSMSBorel$ such that $\mu(\partial A)=0$.
\end{enumerate} 
\end{weakHash_implies_weak_on_spheres2}
The second one confirms that $\weakHashDistancePlain$ is indeed the counterpart of $\prohorovDistancePlain$, that is $\boundedlyFiniteMeasures{\anyCSMS}$ inherits the completeness and separability properties of $\anyCSMS$ under the metric $\weakHashDistancePlain$. This second result also provides us with a characterisation of the Borel $\sigma$-algebra $\borelSetsOf{\boundedlyFiniteMeasures{\anyCSMS}}$.
% Theorem: the realisations space of point processes is a complete separable metric space
\newtheorem{mppSpace_is_CSMS}[weakHash_implies_weak_on_spheres2]{Theorem}
\begin{mppSpace_is_CSMS}[Metric properties of $\boundedlyFiniteMeasures{\anyCSMS}$] \label{thm:mppSpace_is_CSMS}
\hspace{1cm}
\begin{enumerate}[label=\textup{(\roman*)}]
\item The space $\boundedlyFiniteMeasures{\anyCSMS}$ is a complete separable metric space when it is equipped with the distance function $\weakHashDistancePlain$.
\item The corresponding Borel $\sigma$-algebra $\borelSetsOf{\boundedlyFiniteMeasures{\anyCSMS}}$ is the smallest $\sigma$-algebra that makes all mappings $\Phi_A:\boundedlyFiniteMeasures{\anyCSMS} \rightarrow \integers\cup\{\infty\}$, $A\in\anyCSMSBorel$, measurable, where $\Phi_A(\xi)=\xi(A)$.
\end{enumerate}
\end{mppSpace_is_CSMS}
Theorem \ref{thm:weakHash_implies_weak_on_spheres2} and Theorem \ref{thm:mppSpace_is_CSMS} in this paper are Proposition A2.6.II and Theorem A2.6.III in \citet[p.~403--405]{daleyVereJonesVolume1}, respectively.

Regarding the motivation of this article,
the metric space $(\boundedlyFiniteMeasures{\anyCSMS},\weakHashDistancePlain)$ is a stepping stone to the theory of point processes as presented  by \citet{daleyVereJonesVolume2}, who define a point process as a random element in $\boundedlyFiniteMeasures{\anyCSMS}$.
The present research was in fact triggered by the work of \citet{morariu:2017:hybrid} who study the existence and uniqueness of marked point processes defined via their intensity. Since the above theorems are crucial in their framework and proofs, the present author examined them carefully, which resulted in this article.

We now turn to the precise purpose of this paper. To argue that the integrand in \eqref{eq:weak_hash_distance_definition} is measurable and prove the  above properties of the metric $\weakHashDistancePlain$, \citet[p.~403--405]{daleyVereJonesVolume1} assume that $d(\mu^{(r)},\nu^{(r)})$ is non-decreasing as a function of $r\in\realLineNonNegative$. However, this does not seem true as suggested by the following counterexample.
\theoremstyle{definition}
\newtheorem{counterexample}[weakHash_implies_weak_on_spheres2]{Example}
\begin{counterexample}
	Set $\anyCSMS =\realLine$, $x_0=0$, $\mu=\delta_{0}$ and $\nu=\delta_{0.5}$, where, for any $x\in\anyCSMS$, $\delta_x$ denotes the Dirac measure at $x$. Then, as long as $r<0.5$, $d(\mu^{(r)},\nu^{(r)})=1$. However, as soon as $r>0.5$, $d(\mu^{(r)},\nu^{(r)})=0.5$. 
\end{counterexample}
\theoremstyle{plain}
Consequently, our goal is to clarify the original proofs of Theorems \ref{thm:weakHash_implies_weak_on_spheres2} and \ref{thm:mppSpace_is_CSMS} given in \citet{daleyVereJonesVolume1} by addressing specifically the parts that rely on the assumed monotonicity of $d(\mu^{(r)},\nu^{(r)})$. Note that \citet{daleyVereJonesVolume1} consider the larger space $\mathcal{M}_{\anyCSMS}^{\#}$ of boundedly finite measures, i.e., not necessarily integer-valued. The proofs we develop here (except in Section \ref{sec:metric_well_defined}) are specialised to the subspace $\boundedlyFiniteMeasures{\anyCSMS}$ and take advantage of the discrete nature of its elements. Besides, we should add that an alternative metrization of $\mathcal{M}_{\anyCSMS}^{\#}$, leading to the same properties, is presented in \citet[Section 4.1, p.~111--117]{kallenberg2017random}. According to \citet[Historical and bibliographical notes, p.~638]{kallenberg2017random}, this extension from totally finite measures to boundedly finite measures under this alternative metric was first developed by \citet{matthes:1974:inifinitelyDivisiblePPs}.

The paper is organised as follows. Section \ref{sec:preliminaries} gives some preliminary results on the Prohorov metric. Section \ref{sec:metric_well_defined} shows that the distance function in \eqref{eq:weak_hash_distance_definition} is well-defined. Section \ref{sec:convergence} deals with the proof of Theorem \ref{thm:weakHash_implies_weak_on_spheres2}. Sections \ref{sec:completeness_separability} and \ref{sec:sigma_algebra_characterisation} address the proof of Theorem \ref{thm:mppSpace_is_CSMS}.
\theoremstyle{definition}
\newtheorem{how_to_read}[weakHash_implies_weak_on_spheres2]{Remark}
\begin{how_to_read}
	We would like to stress that this paper focuses on the parts of the original proofs that assume that $r\mapsto d(\mu^{(r)},\nu^{(r)})$ is non-decreasing (with the exception of Section \ref{sec:sigma_algebra_characterisation}). Our main objective is to find alternative arguments for these parts specifically. To understand the proofs of Theorems \ref{thm:weakHash_implies_weak_on_spheres2} and \ref{thm:mppSpace_is_CSMS} in their entirety and the details of the other parts that are not treated here, we refer the reader to the original text  \citep[p.~403--405]{daleyVereJonesVolume1}.
\end{how_to_read}
\theoremstyle{plain}

%%%%%%%%%%%%%%%%%%%%%%%%%%%%%%%%%%%%%%%%%%%%%%%%%%%%%
\section{Preliminaries on the Prohorov metric} \label{sec:preliminaries}

As the Prohorov metric $\prohorovDistancePlain$ is the main building block of the $\weakHash$-distance $\weakHashDistancePlain$, it is not surprising that we need to study its behaviour. In particular, we will apply the following lemmas.
% Lemma: upper bound on Prohorov distance
\newtheorem{bound_prohorov}[weakHash_implies_weak_on_spheres2]{Lemma}
\begin{bound_prohorov}[] \label{lem:bound_prohorov}
Let $\mu\in\mathcal{M}_{\mathcal{X}}^{\#}$ and $p,r\in\realLineNonNegative$ such that $p\leq r$. Then $\prohorovDistance{\mu^{(p)}}{\mu^{(r)}}	\leq \mu(S_{r}\setminus S_{p})$.
\end{bound_prohorov}
\begin{proof}
	Let $\varepsilon > \mu(S_{r}\setminus S_{p})$. Let $F\in\borelSetsOf{\anyCSMS}$ be a closed set. Then, clearly
	\begin{equation*}
		\mu^{(p)}(F)=\mu(F\cap S_{p}) \leq \mu (F^{\varepsilon}\cap S_{r}) + \varepsilon = \mu^{(r)}(F^{\varepsilon}) + \varepsilon.
	\end{equation*}
	Moreover, we have that
	\begin{align*}
		\mu^{(r)}(F) &= \mu(F\cap S_{p}) + \mu(F\cap S_{r}\setminus S_{p}) \\
		&\leq  \mu^{(p)}(F) + \mu(S_{r}\setminus S_{p}) \\
		&\leq \mu^{(p)}(F^\varepsilon) + \varepsilon.
	\end{align*}
	This means exactly that $\prohorovDistance{\mu^{(p)}}{\mu^{(r)}}	\leq \mu(S_{r}\setminus S_{p})$ by definition of the Prohorov distance $\prohorovDistancePlain$.
\end{proof}
% Lemma: Prohorov distance when point in band
\newtheorem{bound_prohorov_lower_point}[weakHash_implies_weak_on_spheres2]{Lemma}
\begin{bound_prohorov_lower_point}[] \label{lem:bound_prohorov_lower_point}
Let $\mu,\nu\in\boundedlyFiniteMeasures{\anyCSMS}$ such that $\mu(\anyCSMS)<\infty$, $\nu(\anyCSMS)<\infty$. Let $\underbar{r},\bar{r},\varepsilon\in\realLinePositive$ such that $\underbar{r}<\bar{r}$ and $\varepsilon<(\bar{r}-\underline{r})/2<1$. If $\mu(B_{\bar{r}}\setminus B_{\underline{r}})= 0$ and $\nu(B_{\bar{r}-\varepsilon} \setminus B_{\underline{r}+\varepsilon})>0$, then $d(\mu,\nu) \geq \varepsilon$.
\end{bound_prohorov_lower_point}
\begin{proof}
	Let $0\leq \delta <\varepsilon$ and $u\in B_{\bar{r}-\varepsilon}\setminus B_{\underline{r}+\varepsilon}$ such that $\nu(\{u\})\geq 1$. Then, we have that
	\begin{equation*}
		\nu(\{u\}) \geq 1 >  \delta = \mu(\{u\}^{\delta}) + \delta,
	\end{equation*}
	which implies that $d(\mu,\nu) \geq \delta$ by definition of the Prohorov distance. As a consequence, we have that $d(\mu,\nu)\geq \varepsilon$.
\end{proof}
% Lemma: lower bound on Prohorov distance
\newtheorem{bound_prohorov_lower}[weakHash_implies_weak_on_spheres2]{Lemma}
\begin{bound_prohorov_lower}[] \label{lem:bound_prohorov_lower}
Let $r\in\realLineNonNegative$ and $\mu,\nu\in\boundedlyFiniteMeasures{\anyCSMS}$. Then $\prohorovDistance{\mu^{(r)}}{\nu^{(r)}}	\geq |\mu(B_{r}) - \nu(B_{r})|$.
\end{bound_prohorov_lower}
\begin{proof}
	Without loss of generality, we can assume that $\mu(B_{r})> \nu(B_{r})$. Let $\varepsilon\in [0,\mu(B_{r}) -\nu(B_{r}))$ and let $\delta \in [0, \mu(B_{r}) -\nu(B_{r}) - \varepsilon)$. By Proposition A2.2.II in \citet[p.~386]{daleyVereJonesVolume1}, there exists a closed set $F\subset B_{r}$ such that
	$\mu^{(r)}(B_{r}\setminus F)<\delta$.
	Then, we have that
	\begin{align*}
		\mu^{(r)}(F) &= \mu^{(r)}(B_{r}) - \mu^{(r)}(B_{r}\setminus F) > \mu^{(r)}(B_{r}) - \delta \\
		&> \mu^{(r)}(B_{r}) + \varepsilon + \nu(B_{r}) -\mu(B_{r}) \geq \nu^{(r)}(F^{\varepsilon}) + \varepsilon.
	\end{align*}
	Again, this implies that $\prohorovDistance{\mu^{(r)}}{\nu^{(r)}}	\geq |\mu(B_{r}) - \nu(B_{r})|$ by definition of the Prohorov distance.
\end{proof}

%%%%%%%%%%%%%%%%%%%%%%%%%%%%%%%%%%%%%%%%%%%%%%%%%%%%%
\section{The metric $\weakHashDistancePlain$ is well-defined} \label{sec:metric_well_defined}

In this section, we address the proof in \citet[p.~403]{daleyVereJonesVolume1} that shows that $\weakHashDistancePlain$ is indeed a well-defined metric.
We have to check that the integral in \eqref{eq:weak_hash_distance_definition} is well-defined and, in particular, that $r\mapsto d(\mu^{(r)},\nu^{(r)})$ is measurable. To achieve this, it suffices to notice that this function is actually piecewise constant since $\mu$ and $\nu$ are atomic with finitely many atoms in any bounded set. In fact, for any $R\in\realLinePositive$, as $r$ goes from $0$ to $R$, the restricted measures $\mu^{(r)}$ and $\nu^{(r)}$ change only a finite number of times and so does $d(\mu^{(r)},\nu^{(r)})$. The other arguments in \citet[p.~403]{daleyVereJonesVolume1} are then enough to obtain that $\weakHashDistancePlain$ satisfies all the conditions of a distance function.

As a side note, for the general case where $\mu,\nu\in\mathcal{M}_{\anyCSMS}^{\#}$, we can prove that $r\mapsto d(\mu^{(r)},\nu^{(r)})$ is  measurable by showing that it is of finite variation.
% Lemma: bound on Prohorov distance is of bounded variation
\newtheorem{prohorov_finite_variation}[weakHash_implies_weak_on_spheres2]{Proposition}
\begin{prohorov_finite_variation}[] \label{lem:prohorov_finite_variation}
Let $\mu,\nu\in\mathcal{M}_{\mathcal{X}}^{\#}$ and $R\in\realLineNonNegative$. Then, as a function of $r\in\realLineNonNegative$, the variation of $\prohorovDistance{\mu^{(r)}}{\nu^{(r)}}$ over $[0,R]$ is bounded by $\mu(S_{R}) + \nu(S_{R})$. In particular, $r\mapsto   \prohorovDistance{\mu^{(r)}}{\nu^{(r)}}$ is of bounded variation and, thus, measurable.
\end{prohorov_finite_variation}
\begin{proof}
	Let $r\in\realLineNonNegative$ and $\delta>0$. Applying the triangle inequality to the Prohorov distance, we obtain the following two inequalities:
	\begin{align*}
		\prohorovDistance{\mu^{(r+\delta)}}{\nu^{(r+\delta)}} &\leq \prohorovDistance{\mu^{(r+\delta)}}{\mu^{(r)}} + \prohorovDistance{\mu^{(r)}}{\nu^{(r)}} + \prohorovDistance{\nu^{(r)}}{\nu^{(r+\delta)}}\,; \\
		\prohorovDistance{\mu^{(r)}}{\nu^{(r)}} &\leq \prohorovDistance{\mu^{(r)}}{\mu^{(r+\delta)}} + \prohorovDistance{\mu^{(r+\delta)}}{\nu^{(r+\delta)}} + \prohorovDistance{\nu^{(r+\delta)}}{\nu^{(r)}}.
	\end{align*}
	This implies that
	\begin{equation*}
		|\prohorovDistance{\mu^{(r+\delta)}}{\nu^{(r+\delta)}} - \prohorovDistance{\mu^{(r)}}{\nu^{(r)}}| \leq \prohorovDistance{\mu^{(r)}}{\mu^{(r+\delta)}} + \prohorovDistance{\nu^{(r)}}{\nu^{(r+\delta)}}.
	\end{equation*}
	Using Lemma \ref{lem:bound_prohorov}, we can go further and get that
	\begin{equation*}
		|\prohorovDistance{\mu^{(r+\delta)}}{\nu^{(r+\delta)}} - \prohorovDistance{\mu^{(r)}}{\nu^{(r)}}| \leq \mu(S_{r+\delta}) - \mu(S_{r}) + \nu(S_{r+\delta}) - \nu(S_{r}).
	\end{equation*}
	Since $\mu(S_{r})$ and $\nu(S_{r})$ are non-decreasing in $r$ and always finite (because $\mu$ and $\nu$ are boundedly finite), they are of bounded variation, which concludes the proof.
\end{proof}

%%%%%%%%%%%%%%%%%%%%%%%%%%%%%%%%%%%%%%%%%%%%%%%%%%%%%
\section{Characterisation of convergence in the $\weakHash$-topology} \label{sec:convergence}

In this section, we address the proof of Theorem \ref{thm:weakHash_implies_weak_on_spheres2}, which characterises the convergence of boundedly finite integer-valued measures.
% Proof of the convergence characterisation theorem
\begin{proof}[Proof of Theorem \ref{thm:weakHash_implies_weak_on_spheres2}]
We only need to show the implication (i)$\implies$(iii) as this seems to be the only part in \citet[p.~403]{daleyVereJonesVolume1} assuming that $d(\mu^{(r)},\nu^{(r)})$ is non-decreasing in $r\in\realLineNonNegative$. The rest of the proof of Proposition A2.6.II in \citet[p.~403-404]{daleyVereJonesVolume1} can be used to show that (iii)$\implies$(ii)$\implies$(iv)$\implies$(i).

Let $n\in\integers$ and $\underline{r}_n, \bar{r}_n\in\realLineNonNegative$ such that $n<\underline{r}_n < \bar{r}_n < n+1$ and $\mu(B_{\bar{r}_n}\setminus B_{\underline{r}_n})=0$. Let $0<\varepsilon<(\bar{r}_n-\underline{r}_n)/2$. By contradiction, assume that for any $K\in\integers$, there exists $k>K$ such that $\mu_{k}(B_{\bar{r}_n-\varepsilon}\setminus B_{\underline{r}_n+\varepsilon} )\geq 1$. Then, by Lemma \ref{lem:bound_prohorov_lower_point}, there exists a subsequence $\collection{k_{p}}{p\in\integers}$ such that $\prohorovDistance{\mu^{(r)}_{k_p}}{\mu^{(r)}}\geq \varepsilon$ for all $r\geq n+1$, $p\in\integers$. Thus, along this subsequence, we must have that
\begin{equation*}
	\weakHashDistance{\mu_{k_p}}{\mu} = \int_{0}^{\infty}e^{-r}\frac{d(\mu^{(r)}_{k_p},\mu^{(r)})}{1+d(\mu^{(r)}_{k_p},\mu^{(r)})}dr \geq \int_{n+1}^{\infty}e^{-r}\frac{\varepsilon}{1+\varepsilon}dr = \frac{\varepsilon}{1+\varepsilon}e^{-n-1}>0,
\end{equation*}
which contradicts the assumption that $\weakHashDistance{\mu_{k}}{\mu}\rightarrow 0$ as $k\rightarrow\infty$. As a consequence, there exists a $K\in\integers$ such that, for all $k\geq K$, $\mu_{k}(S_{\bar{r}_n-\varepsilon}\setminus S_{\underline{r}_n+\varepsilon} )=0$. This means that, for all $k\geq K$, both $\mu_k$ and $\mu$ do not have atoms in $S_{\bar{r}_n-\varepsilon}\setminus S_{\underline{r}_n+\varepsilon}$, whence there is some constant $d_{k}\in\realLineNonNegative$ such that $\prohorovDistance{\mu^{(r)}_k}{\mu^{(r)}}=d_{k}$ for all $r\in(\underline{r}_n+\varepsilon,\bar{r}_n-\varepsilon)$. This implies that, for all $k\geq K$,
\begin{equation*}
	\weakHashDistance{\mu_{k}}{\mu} = \int_{0}^{\infty}e^{-r}\frac{d(\mu^{(r)}_{k},\mu^{(r)})}{1+d(\mu^{(r)}_{k},\mu^{(r)})}dr \geq \int_{\underline{r}_n+\varepsilon}^{\bar{r}_n-\varepsilon}e^{-r}\frac{d_{k}}{1+d_{k}}dr \geq \frac{d_{k}}{1+d_{k}}e^{-\underline{r}_n-\varepsilon}(1-e^{-(\bar{r}_n - \underline{r}_n -2\varepsilon )}),
\end{equation*}
and, thus, $d_{k}\rightarrow 0$ as $k\rightarrow\infty$. If we set $r_{n}=(\underline{r}_n + \bar{r}_n)/2$, we finally have that $\prohorovDistance{\mu^{(r_n)}_k}{\mu^{(r_n)}}\rightarrow 0$ as $k\rightarrow\infty$.
\end{proof}

%%%%%%%%%%%%%%%%%%%%%%%%%%%%%%%%%%%%%%%%%%%%%%%%%%%%%
\section{Completeness and separability of $\boundedlyFiniteMeasures{\anyCSMS}$} \label{sec:completeness_separability}

In this section, we address the proof of the first part of Theorem \ref{thm:mppSpace_is_CSMS}, which states that $\boundedlyFiniteMeasures{\anyCSMS}$ is complete and separable when it is endowed with the $\weakHash$-metric $\weakHashDistancePlain$.

\subsection{Completeness}

To begin with, we show that if a sequence $\collection{\mu_k}{k\in\integers}$ in $(\boundedlyFiniteMeasures{\anyCSMS}, \weakHashDistancePlain)$ is Cauchy, then the restrictions along an increasing sequence of balls are also Cauchy for the Prohorov metric $\prohorovDistancePlain$.
% Proposition: Cauchy implies Cauchy
\newtheorem{cauchy_implies_cauchy}[weakHash_implies_weak_on_spheres2]{Proposition}
\begin{cauchy_implies_cauchy}[] \label{prop:cauchy_implies_cauchy}
Let $\collection{\mu_k}{k\in\integers}$	be a Cauchy sequence in $\boundedlyFiniteMeasures{\anyCSMS}$ for the $\weakHash$-metric $\weakHashDistancePlain$. Then, there exists an increasing sequence $\collection{r_n}{n\in\integers}$ in $\realLinePositive$ with $r_{n}\rightarrow\infty$ as $n\rightarrow\infty$ such that, for each $n\in\integers$, $\collection{\mu_k^{(r_n)}}{k\in\integers}$ is a Cauchy sequence in $\mathcal{M}_{\anyCSMS}$ for the Prohorov metric $\prohorovDistancePlain$.
\end{cauchy_implies_cauchy}
\begin{proof}
	\textbf{Step 1.} We show that $\mu_k(B_{r})$ is bounded in $k\in\integers$ for all $r\in\realLineNonNegative$. By contradiction, assume that this is not the case. Then, there exists a subsequence such that $\mu_{k_p}(B_{r})\rightarrow\infty$. Along this subsequence, for $p$ large enough and any fixed $q\in\integers$, we have that
	\begin{align*}
		\int_{r}^{r+1}e^{-s}\frac{\prohorovDistance{\mu_{k_p}^{(s)}}{\mu_{k_q}^{(s)}}}{1+\prohorovDistance{\mu_{k_p}^{(s)}}{\mu_{k_q}^{(s)}}}ds &\geq 
		\int_{r}^{r+1}e^{-s}\frac{|\mu_{k_p}(B_s) - \mu_{k_q}(B_{s})|}{1+|\mu_{k_p}(B_s) - \mu_{k_q}(B_{s})|}ds \\
		&\geq \int_{r}^{r+1}e^{-s}\frac{\mu_{k_p}(B_r) - \mu_{k_q}(B_{r+1})}{1+\mu_{k_p}(B_r) - \mu_{k_q}(B_{r+1})}ds \rightarrow e^{-r}(1-e^{-1}),\quad p\rightarrow\infty,
	\end{align*}
	where we used Lemma \ref{lem:bound_prohorov_lower} and the fact that $\mu_{k_p}(B_{s})$ and $\mu_{k_q}(B_{s})$ are non-decreasing in $s$. But this is incompatible with the Cauchy assumption on  $\collection{\mu_k}{k\in\integers}$. Indeed, let $\varepsilon<e^{-r}(1-e^{-1})$. Then, the Cauchy assumption implies that there exists $K\in\integers$ such that, for all $k,k'\geq K$,
	\begin{align*}
	\weakHashDistance{\mu_{k}}{\mu_{k'}} = \int_{0}^{\infty}e^{-s} \frac{\prohorovDistance{\mu_{k}^{(s)}}{\mu_{k'}^{(s)}}}{1+\prohorovDistance{\mu_{k}^{(s)}}{\mu_{k'}^{(s)}}}ds	\leq \varepsilon.	
	\end{align*}
	But then, for $p,q\in\integers$ large enough, we must have that
	\begin{align*}
		\varepsilon \geq \int_{0}^{\infty}e^{-s} \frac{\prohorovDistance{\mu_{k_p}^{(s)}}{\mu_{k_q}^{(s)}}}{1+\prohorovDistance{\mu_{k_p}^{(s)}}{\mu_{k_q}^{(s)}}}ds \geq \int_{r}^{r+1}e^{-s} \frac{\prohorovDistance{\mu_{k_p}^{(s)}}{\mu_{k_q}^{(s)}}}{1+\prohorovDistance{\mu_{k_p}^{(s)}}{\mu_{k_q}^{(s)}}}ds > \varepsilon.
	\end{align*}
	
	\textbf{Step 2.} Let $n\in\integers$. We show that for $k,p\in\integers$ large enough, there is a  subinterval of $[n,n+1]$ on which the functions $r\mapsto \prohorovDistance{\mu_k^{(r)}}{\mu_p^{(r)}}$ are constant.  Define $M:=\sup_{k\in\integers}\mu_{k}(B_{n+1})$, which is finite by the first step and can be understood as a bound on the number of points in the ball $B_{n+1}$ among all measures $\mu_{k}$. Let $\varepsilon_1, \varepsilon_2 \in\realLinePositive$ such that $\varepsilon_1 < \varepsilon_2 < 1/2(M+1)$ and $\varepsilon_1 < \varepsilon_2 e^{-n-1} / (1+\varepsilon_2)$. Let $K\in\integers$ such that, for all $k,k'\geq K$, $\weakHashDistance{\mu_k}{\mu_{k'}} \leq \varepsilon_1$ (Cauchy assumption). Since $\mu_{K}(B_{n+1}\setminus B_n) \leq M$, we can find $\underline{r}_n,\overline{r}_n \in (n,n+1)$ such that $\mu_K(B_{\overline{r}_n}\setminus B_{\underline{r}_n})=0$ and $\overline{r}_n-\underline{r}_n \geq 1/(M+1)$. Now, by contradiction, assume that for some $p>K$, we have $\mu_p(B_{\overline{r}_n-\varepsilon_2} \setminus B_{\underline{r}_n+\varepsilon_2} )\geq 1$. Then, using Lemma \ref{lem:bound_prohorov_lower_point}, we obtain that
	\begin{align*}
		\varepsilon_1 \geq \weakHashDistance{\mu_K}{\mu_p} = \int_{0}^{\infty}e^{-r} \frac{\prohorovDistance{\mu_{K}^{(r)}}{\mu_{p}^{(r)}}}{1+\prohorovDistance{\mu_{K}^{(r)}}{\mu_{p}^{(r)}}}dr \geq \int_{n+1}^{\infty}e^{-r} \frac{\prohorovDistance{\mu_{K}^{(r)}}{\mu_{p}^{(r)}}}{1+\prohorovDistance{\mu_{K}^{(r)}}{\mu_{p}^{(r)}}}dr \geq \frac{\varepsilon_2}{1+\varepsilon_2}e^{-n-1},
	\end{align*}
	which contradicts the original assumption on $\varepsilon_1$ and $\varepsilon_2$. As a consequence, for all $k\geq K$, we have that $\mu_k(B_{\overline{r}_n-\varepsilon_2}\setminus B_{\underline{r}_n+\varepsilon_2}) =0$, which implies that $r\mapsto \prohorovDistance{\mu_p^{(r)}}{\mu_q^{(r)}}$ is constant on $(\underline{r}_n+\varepsilon_2, \overline{r}_n-\varepsilon_2)$ for all $p,q\geq K$.
	
	\textbf{Step 3.} We finally show that when $r_n=:(\underline{r}_n + \overline{r}_n)/2$, $\collection{\mu_k^{(r_n)}}{k\in\integers}$ is a Cauchy sequence for the Prohorov metric $\prohorovDistancePlain$. Let $\varepsilon>0$ and set $\delta:=(\overline{r}_n-\underline{r}_n-2\varepsilon_2)e^{-n-1}\varepsilon/(1+\varepsilon)$.  Let $J\in\integers$ such that, for all $p,q\geq J$, $\weakHashDistance{\mu_k}{\mu_{k'}}\leq \delta$ (Cauchy assumption). Then, for all $p,q\geq K\vee J$, we must have that
	\begin{equation*}
		\delta \geq \int_{\underline{r}_n+\varepsilon_2}^{\overline{r}_n-\varepsilon_2}e^{-r}\frac{d_{pq}}{1+d_{pq}}dr \geq \frac{d_{pq}}{1+d_{pq}}(\overline{r}_n-\underline{r}_n-2\varepsilon_2)e^{-n-1},
	\end{equation*}
	where $d_{pq}:=\prohorovDistance{\mu_p^{(r_n)}}{\mu_q^{(r_n)}}$, and which implies
	\begin{equation*}
		d_{pq} \leq \frac{\delta}{(\overline{r}_n-\underline{r}_n-2\varepsilon_2)e^{-n-1}-\delta}=\frac{1}{\frac{1+\varepsilon}{\varepsilon}-1} =\varepsilon. \qedhere
	\end{equation*}
\end{proof}

Reusing a part of the proof of Theorem A2.6.III in \citet[p.~404]{daleyVereJonesVolume1}, the above proposition implies that $\boundedlyFiniteMeasures{\anyCSMS}$ is complete. Still, we would like to mention some points that could deserve a bit more detail. First, one needs to ensure that the limit of each Cauchy sequence $\collection{\mu_k^{(r_n)}}{k\in\integers}$ in Proposition \ref{prop:cauchy_implies_cauchy} is still integer-valued. This can be done by adapting the proof of Lemma 9.1.V in \citet[p.~6]{daleyVereJonesVolume2}. Second, if we denote the limit of $\collection{\mu_k^{(r_n)}}{k\in\integers}$ by $\nu_{n}$, we can show that $\nu_{m}^{(r_n)}=\nu_{n}$ when $n<m$ (i.e., the sequence of measures $\collection{\nu_n}{n\in\integers}$ is consistent) by using Theorem A2.3.II.(iv) in  \citet[p.~391]{daleyVereJonesVolume1} and the fact that $\nu_m(\partial B_{r_n})=0$.
Third, to show that
	$\mu(\cdot):=\lim_{n\rightarrow\infty} \nu_n(\cdot)$
is continuous from below, one can use the fact that $\lim_{i\rightarrow\infty}\lim_{j\rightarrow\infty} a_{ij} = \lim_{j\rightarrow\infty}\lim_{i\rightarrow\infty} a_{ij}$ for any double sequence $(a_{ij})$ that is non-decreasing in both $i$ and $j$.

\subsection{Separability}

Next, we prove that the space of boundedly finite integer-valued measures $\boundedlyFiniteMeasures{\anyCSMS}$ is separable. We wish to show that there exists a countable set in $\boundedlyFiniteMeasures{\anyCSMS}$ that can approximate well-enough any element of $\boundedlyFiniteMeasures{\anyCSMS}$. Let $\mathcal{D}_{\anyCSMS}$ be the separability set of $\anyCSMS$. It seems natural to expect that the set of totally finite (hence with a finite number of atoms) integer-valued measures with atoms only in $\mathcal{D}_{\anyCSMS}$ is a good candidate. We denote this set by $\mathcal{D}_{\mathcal{N}}$. Notice that one can define an injection between $\mathcal{D}_{\mathcal{N}}$ and the finite subsets of $\integers^2$. For example, the Dirac measure with mass $n\in\integers$ at the $m^{th}$ element of $\mathcal{D}_{\anyCSMS}$ can be represented by the set $\{(m, n)\}$. Since the finite subsets of a countable set form a countable set, we have that $\mathcal{D}_{\mathcal{N}}$ is countable. The following proposition coupled with a part of the proof of Theorem A2.6.III in \citet[p.~404]{daleyVereJonesVolume1} allows to conclude that $\mathcal{D}_{\mathcal{N}}$ is indeed a separability set for $\boundedlyFiniteMeasures{\anyCSMS}$.
% Proposition: Separability of the space of boundedly finite integer-valued measures.
\newtheorem{separability}[weakHash_implies_weak_on_spheres2]{Proposition}
\begin{separability}[] \label{prop:separability}
Let $\mu\in\boundedlyFiniteMeasures{\anyCSMS}$ and $R,\varepsilon\in\realLinePositive$.	Then, there exists $\tilde{\mu}\in\mathcal{D}_{\mathcal{N}}$ such that
\begin{equation*}
	\int_{0}^{R}e^{-r}\frac{\prohorovDistance{\mu^{(r)}}{\tilde{\mu}^{(r)}}}{1+\prohorovDistance{\mu^{(r)}}{\tilde{\mu}^{(r)}}}dr \leq \varepsilon.
\end{equation*}
\end{separability}
\begin{proof}
	Let $\collection{u_{n}}{n\in\{1,\ldots,N\}}$ be the atoms of $\mu$ in $B_R$ where $N\in\integers$ is their total number and let $\collection{w_{n}}{n\in\{1,\ldots,N\}}$ be their corresponding weights. Let $\varepsilon_1 >0$ such that $B_{\varepsilon_1}(u_{n})\subset B_{R}$ for all $n=1,\ldots,N$. Let $0\leq r_{1}<\ldots<r_{N'}<R$ be the radii at which the atoms are located where $N'\in\integers$, $N'\leq N$ ($r_1=0$ means that $x_0\in \collection{u_{n}}{n\in\{1,\ldots,N\}}$). Define $\varepsilon_2 :=\frac{1}{2}\min_{n<N'}(r_{n+1}-r_{n})$ and $\varepsilon_3:=\varepsilon/4N'$. Define $\varepsilon_{4}:=\varepsilon/(2c-\varepsilon)$, where $c=1-e^{-R}$, and assume that $\varepsilon < 2c$ (if this is not the case, then the desired inequality already holds no matter $\tilde{\mu}$). Finally, set $\delta:=\min(\varepsilon_{1},\varepsilon_{2},\varepsilon_{3},\varepsilon_{4})$ and let $\collection{\tilde{u}_n}{n\in\{1,\ldots,N\}}$ be such that $\tilde{u}_n\in\mathcal{D}_{\anyCSMS}$, $\tilde{u}_n\in B_\delta(u_{n})$, $n=1,\ldots,N$. We will show that $\tilde{\mu}:=\sum_{n=1}^{N}w_{n}\delta_{\tilde{u}_n}$ satisfies the desired inequality.
	
	Let $n=1,\ldots,N'-1$ and $r\in (r_{n}+\delta, r_{n+1}-\delta)$. We can check that $\prohorovDistance{\mu^{(r)}}{\tilde{\mu}^{(r)}} \leq \delta$. Indeed, since $\delta\leq \varepsilon_1$ and $\delta\leq \varepsilon_2$, we have that $u_{i}\in B_r$ if and only if $\tilde{u}_i\in B_r$. As a consequence, for any closed set $A\in\anyCSMSBorel\cap B_{r}$, using the fact that $\tilde{u}_i\in B_\delta(u_{i})$, we have that
	\begin{equation*}
		\mu^{(r)}(A)=\mu(A) \leq \tilde{\mu}(A^{\delta}\cap B_r)=\tilde{\mu}^{(r)}(A^{\delta}) \quad \mbox{and} \quad \tilde{\mu}^{(r)}(A)=\tilde{\mu}(A)\leq \mu(A^{\delta}\cap B_r)=\mu^{(r)}(A^{\delta}),
	\end{equation*}
	which means that $\prohorovDistance{\mu^{(r)}}{\tilde{\mu}^{(r)}} \leq \delta$. Similarly, we also have that $\prohorovDistance{\mu^{(r)}}{\tilde{\mu}^{(r)}} \leq \delta$ for all $r\in[0,0\vee (r_{1}-\delta))$ and all $r\in(r_{N'}+\delta ,R]$. Using this bound on the Prohorov distance between the restrictions, we obtain that
	\begin{align*}
		\int_{0}^{R}\frac{e^{-r}\prohorovDistance{\mu^{(r)}}{\tilde{\mu}^{(r)}}}{1+\prohorovDistance{\mu^{(r)}}{\tilde{\mu}^{(r)}}}dr = &
		\left(\int_{0}^{0\vee (r_{1}-\delta)} + \sum_{n=1}^{N'}\int_{(r_{n}-\delta)\vee 0}^{r_{n}+\delta} + \sum_{n=1}^{N'-1}\int_{r_{n}+\delta}^{r_{n+1}-\delta} + \int_{r_{N'}+\delta}^{R} \right)
		\frac{e^{-r}\prohorovDistance{\mu^{(r)}}{\tilde{\mu}^{(r)}}}{1+\prohorovDistance{\mu^{(r)}}{\tilde{\mu}^{(r)}}}dr \\
		&\leq \int_{0}^{R}e^{-r}\frac{\delta}{1+\delta}dr + \sum_{n=1}^{N'}\int_{(r_{n}-\delta)\vee 0}^{r_{n}+\delta}e^{-r}\frac{\prohorovDistance{\mu^{(r)}}{\tilde{\mu}^{(r)}}}{1+\prohorovDistance{\mu^{(r)}}{\tilde{\mu}^{(r)}}}dr \\
		&\leq (1-e^{-R})\frac{\delta}{1+\delta} + 2\delta N'  \leq (1-e^{-R})\frac{\varepsilon_4}{1+\varepsilon_4} + 2\varepsilon_3 N'=\frac{\varepsilon}{2} + \frac{\varepsilon}{2} = \varepsilon. \qedhere
	\end{align*}
\end{proof}

%%%%%%%%%%%%%%%%%%%%%%%%%%%%%%%%%%%%%%%%%%%%%%%%%%%%%
\section{Characterisation of the $\sigma$-algebra $\boundedlyFiniteMeasuresBorel{\anyCSMS}$} \label{sec:sigma_algebra_characterisation}

This section proves the second part of Theorem \ref{thm:mppSpace_is_CSMS}. We show that all mappings $\Phi_A: \xi\mapsto\xi(A)$, $\xi\in\boundedlyFiniteMeasures{\anyCSMS}$, $A\in\anyCSMSBorel$, are measurable with respect to the Borel $\sigma$-algebra $\boundedlyFiniteMeasuresBorel{\anyCSMS}$ and that $\boundedlyFiniteMeasuresBorel{\anyCSMS}$ is actually generated by all these mappings. This property is very useful to check the measurability of functionals on $\boundedlyFiniteMeasures{\anyCSMS}$, like for example Hawkes functionals, as demonstrated in \citet{morariu:2017:hybrid}. Our proof is different from the original one in \citet[p.~405]{daleyVereJonesVolume1} as we identify a convenient basis for the $\weakHash$-hash topology (Proposition \ref{prop:basis_weak_hash_topology}). Note however that this last result is directly inspired by Proposition A2.5.I in \citet[p.~398]{daleyVereJonesVolume1}, where three different bases for the weak topology on $\mathcal{M}_{\anyCSMS}$ are given. Besides, our proof of Theorem \ref{thm:mppSpace_is_CSMS}.(ii) shows explicitly why the mapping $\Phi_A$ is $\boundedlyFiniteMeasuresBorel{\anyCSMS}$-measurable when $A$ is a bounded closed set.
% Proposition: characterisation sigma algebra boundedly finite measures
\newtheorem{basis_weak_hash_topology}[weakHash_implies_weak_on_spheres2]{Proposition}
\begin{basis_weak_hash_topology}[] \label{prop:basis_weak_hash_topology}
Consider the family of sets
\begin{align} \label{eq:basis_weak_hash_topology}
	\{\xi\in\boundedlyFiniteMeasures{\anyCSMS}\,:\, &\xi(F_i)<\mu(F_i) + \varepsilon \mbox{ for } i=1,\ldots,m, \\
	& |\xi(\overline{B}_{r_j}) - \mu(\overline{B}_{r_j})|<\varepsilon \mbox{ and } \xi(\partial {B}_{r_j}) = 0 \mbox{ for } j=1,\ldots,n \}, \nonumber
\end{align}
where $\mu\in\boundedlyFiniteMeasures{\anyCSMS}$, $\varepsilon\in\realLinePositive$, $m,n\in\integers$, $F_{i}$, $i=1,\ldots,m$, is a bounded closed set of $\anyCSMS$ and $r_j\in\realLinePositive$, $j=1,\ldots,n$, is such that $\mu(\partial {B}_{r_j})=0$. This family forms a basis that generates the $\weakHash$-topology.	
\end{basis_weak_hash_topology}
\begin{proof}
	\textbf{Step 1.} We check that this family is a basis. Let $\mu,\mu'\in\boundedlyFiniteMeasures{\anyCSMS}$, $\varepsilon,\varepsilon'\in\realLinePositive$, let $F_{1},\ldots,F_m$ and  $F'_1,\ldots,F'_{m'}$ be bounded closed sets and let $r_1,\ldots,r_n, r'_1,\ldots, r'_{n'}>0$ such that $\mu(\partial {B}_{r_j})=0$ and $\mu'(\partial {B}_{r'_j})=0$. Consider the sets $A$ and $B$ of the form \eqref{eq:basis_weak_hash_topology} generated by these two collections, respectively, and let $\mu''\in A\cap B$. We will now find a set $C$, again of the form \eqref{eq:basis_weak_hash_topology}, such that $\mu''\in C$ and $C\subset A\cap B$.
	Set the following parameters:
	\begin{align*}
		\delta := \min_{i=1,\ldots,m}\mu(F_i)+\varepsilon -\mu''(F_i) &,\quad \delta' := \min_{i=1,\ldots,m'}\mu'(F'_i)+\varepsilon' -\mu''(F'_i), \\
		\gamma := \min_{j=1,\ldots,n} \varepsilon - |\mu''(\overline{B}_{r_j}) - \mu(\overline{B}_{r_j})| &,\quad \gamma' := \min_{j=1,\ldots,n'} \varepsilon' - |\mu''(\overline{B}_{r'_j}) - \mu'(\overline{B}_{r'_j})|;
	\end{align*}
	and let $\varepsilon'' := \min(\delta,\delta',\gamma,\gamma')$. Now, consider the set
	\begin{align*}
		C:=\{\xi\in\boundedlyFiniteMeasures{\anyCSMS}\,:\, &\xi(F_i)<\mu''(F_i) + \varepsilon'' \mbox{ for } i=1,\ldots,m,\, \xi(F'_i)<\mu''(F'_i) + \varepsilon'' \mbox{ for } i=1,\ldots,m', \\
	& |\xi(\overline{B}_{r_j}) - \mu''(\overline{B}_{r_j})|<\varepsilon'' \mbox{ and } \xi(\partial {B}_{r_j}) = 0 \mbox{ for } j=1,\ldots,n, \\
	& |\xi(\overline{B}_{r'_j}) - \mu''(\overline{B}_{r'_j})|<\varepsilon'' \mbox{ and } \xi(\partial {B}_{r'_j}) = 0 \mbox{ for } j=1,\ldots,n' \}.
	\end{align*}
	Clearly, the set $C$ is of the form \eqref{eq:basis_weak_hash_topology}. We now finally check that $C\subset A \cap B$. Let $\xi\in C$. For all $i=1,\ldots,m$, we have that
	\begin{equation*}
		\xi(F_i) < \mu''(F_i) + \varepsilon'' \leq \mu(F_i) + \varepsilon,
	\end{equation*}
	because $\varepsilon''\leq \mu(F_i)  + \varepsilon - \mu''(F_i)$. For all $j=1,\ldots,n$, we have that
	\begin{equation*}
		|\xi(\overline{B}_{r_j}) - \mu(\overline{B}_{r_j})| \leq |\xi(\overline{B}_{r_j}) - \mu''(\overline{B}_{r_j})| + |\mu''(\overline{B}_{r_j}) - \mu(\overline{B}_{r_j})| < \varepsilon,
	\end{equation*}
	because $\varepsilon'' \leq \varepsilon - |\mu''(\overline{B}_{r_j}) - \mu(\overline{B}_{r_j})|$. Thus, $\xi\in A$. A similar argument yields $\xi\in B$ and so $C\subset A\cap B$.
	
	\textbf{Step 2.} We check that every element of this basis contains an open ball. Consider first any set $A$ of the form \eqref{eq:basis_weak_hash_topology} but for which $n=1$ (only one ball). Let $\delta\in(0,1)$ such that $2\delta<\varepsilon$, $\mu(F_i^{\delta})=\mu(F_i)$ for all $i=1,\ldots,m$, and 
		\begin{equation*}
			\mu\left(\overline{\overline{B}_{r_1}^{\delta}\setminus\overline{B}_{r_1}}^{\delta}\right)=0,
		\end{equation*}
	which means that $\delta$ is chosen small enough such that there are no atoms within a distance $\delta$ of the boundary $\partial{B}_{r_1}$. Let $R\in\realLinePositive$ such that $F_i^\delta \subset B_R$ for all $i=1,\ldots,m$ and such that $r_1 + 2\delta < R$. Consider now the ball $B:=\{\xi\in\boundedlyFiniteMeasures{\anyCSMS}\,:\,\weakHashDistance{\mu}{\xi}<\gamma\}$ where $\gamma:=e^{-R}\delta/(1+\delta)$. Take any $\xi\in B$ and, by contradiction, assume that $\xi(F_i) > \mu(F_i^\delta) + \delta$ for some $i=1,\ldots,m$. Then, this implies that $\prohorovDistance{\xi^{(r)}}{\mu^{(r)}}\geq\delta$ for all $r\geq R$, which in turn implies that
	\begin{equation*}
		\weakHashDistance{\xi}{\mu}\geq \int_R^{\infty}e^{-r}\frac{\delta}{1+\delta}dr = \gamma.
	\end{equation*}
	This contradicts the fact that $\xi\in B$ and, thus, we must have that
	\begin{equation*}
		\xi(F_i) \leq \mu(F_i^\delta) + \delta = \mu(F_i) + \delta < \mu(F_i) + \varepsilon, \quad i=1,\ldots,m.
	\end{equation*}
	The same reasoning holds for the closed sets $\overline{B}_{r_1}$ and $\partial {B}_{r_1}$, finally implying that
	\begin{equation*}
		\xi(\partial{B}_{r_1}) = \mu(\partial{B}_{r_1}) = 0 \quad \mbox{and} \quad \xi(\overline{B}_{r_1}) - \mu(\overline{B}_{r_1}) \leq \delta < \varepsilon.
	\end{equation*}
	To obtain that $\xi\in A$, it remains only to show that $\mu(\overline{B}_{r_1}) - \xi(\overline{B}_{r_1}) < \varepsilon$. Using again the previous reasoning, we also have that
	\begin{equation*}
		\xi(\overline{B}_{r_1}^{\delta}) - \xi(\overline{B}_{r_1})=\xi(\overline{B}_{r_1}^{\delta}\setminus \overline{B}_{r_1}) \leq \xi\left(\overline{\overline{B}_{r_1}^{\delta}\setminus \overline{B}_{r_1}}\right) \leq \mu\left(\overline{\overline{B}_{r_1}^{\delta}\setminus \overline{B}_{r_1}}^{\delta}\right) + \delta =  \delta,
	\end{equation*}
	and also that $\mu(\overline{B}_{r_1})\leq \xi(\overline{B}_{r_1}^\delta)+\delta$.
	This implies the desired inequality
	\begin{equation*}
		\mu(\overline{B}_{r_1}) - \xi(\overline{B}_{r_1}) = \mu(\overline{B}_{r_1}) - \xi(\overline{B}_{r_1}^\delta) + \xi(\overline{B}_{r_1}^\delta) - \xi(\overline{B}_{r_1}) \leq \delta + \delta < \varepsilon,
	\end{equation*}
	and allows us to conclude that the ball $B$ is included in the neighbourhood $A$. Regarding the general case when the set $A$ is defined by multiple balls (i.e., $n>1$), simply view it as an intersection of sets $A_{j}$, where each $A_j$ is defined by one ball (i.e., $m=1$). As shown above, for each $A_j$, we can find an adequate ball with centre $\mu$ and radius $\gamma_j$. Then, the ball with radius $\gamma=\min \gamma_i$ must be included in $A$.
	
	\textbf{Step 3.} We check that every open ball contains an element of this basis. Let $\mu\in\boundedlyFiniteMeasures{\anyCSMS}$, $\varepsilon\in\realLinePositive$ and consider the ball $B:=\{\xi\in\boundedlyFiniteMeasures{\anyCSMS}\,:\,\weakHashDistance{\mu}{\xi} < \varepsilon\}$. Let $R>0$ such that $e^{-R}<\frac{1}{2}\varepsilon$. Let $\rho_1 < \ldots < \rho_N$ be all the radii in $(0,R)$ such that $\mu(\partial{B}_{\rho_j})>0$, $j=1,\ldots,N$. Set also $\rho_0:=0$ and $\rho_{N+1}:=R$. Define $\rho:=\frac{1}{2}\min_{j=1,\ldots,N+1} \rho_j - \rho_{j-1}$, let $\gamma < \varepsilon/8(N+2)$ and set $\delta:=\min(\rho,\gamma)$. Define the bounded closed sets $G_j := \overline{B}_{\rho_{j}-\delta}\setminus B_{\rho_{j-1} + \delta}$ for $j=1,\ldots,N+1$ and notice that $\mu(G_j)=0$. Also, define the radii $r_j := (\rho_{j-1}+\rho_j)/2$, $j=1,\ldots,N+1$. For all $r_j$, reusing the last part of the proof of Proposition A2.5.I in \citet[p.~399]{daleyVereJonesVolume1}, we know that we can find $\tilde{\varepsilon}_j\in(0,1)$ and a finite family of closed bounded sets $F_{1,j},\ldots,F_{m_j,j}$ such that
	\begin{align*}
		A_j := \{\xi\in\boundedlyFiniteMeasures{\anyCSMS}\,:\, &\xi(F_{i,j})<\mu(F_{i,j}) + \tilde{\varepsilon}_j \mbox{ for } i=1,\ldots,m_j,\, |\xi(\overline{B}_{r_j}) - \mu(\overline{B}_{r_j})|<\tilde{\varepsilon}_j \} \\
	& \subset \{ \xi\in\boundedlyFiniteMeasures{\anyCSMS}\,:\, \prohorovDistance{\mu^{(r_j)}}{\xi^{(r_j)}} < c \},
	\end{align*}
	where here we choose $c$ such that $(1-e^{-R})c/(1+c) <\varepsilon/4$.
	Finally, set $\tilde{\varepsilon} = \min \tilde{\varepsilon}_j$ and consider the set
	\begin{align*}
		A := \{\xi\in\boundedlyFiniteMeasures{\anyCSMS}\,:\, &\xi(F_{i,j})<\mu(F_{i,j}) + \tilde{\varepsilon} \mbox{ for } i=1,\ldots,m_j,\, \xi(G_j)<\mu(G_j) + \tilde{\varepsilon},\\
	& |\xi(\overline{B}_{r_j}) - \mu(\overline{B}_{r_j})|<\tilde{\varepsilon} \mbox{ and } \xi(\partial {B}_{r_j}) = 0,\mbox{ for } j=1,\ldots,N+1 \},
	\end{align*}
	which is of the form \eqref{eq:basis_weak_hash_topology} and is such that $A\subset A_j$, $j=1,\ldots,N+1$. For all $\xi\in A$, this implies that $\prohorovDistance{\mu^{(r_j)}}{\xi^{(r_j)}} < c$, $j=1,\ldots,N+1$. This also implies that $\xi(G_j)=0$, and thus $r\mapsto \prohorovDistance{\mu^{(r)}}{\xi^{(r)}}$ is constant on each interval $(\rho_{j-1}+\delta,\rho_j-\delta)$, $j=1,\ldots,N+1$. Noting that $r_j\in(\rho_{j-1}+\delta,\rho_j-\delta)$, it remains to check that
	\begin{align*}
		\weakHashDistance{\mu}{\xi} &< \int_0^R e^{-r}\frac{\prohorovDistance{\mu^{(r)}}{\xi^{(r)}}}{1+\prohorovDistance{\mu^{(r)}}{\xi^{(r)}}}dr + \frac{1}{2}\varepsilon \\
		&<2\delta (N+2) +  (1-e^{-R})\frac{c}{1+c} + \frac{1}{2}\varepsilon < \frac{1}{4}\varepsilon + \frac{1}{4}\varepsilon + \frac{1}{2}\varepsilon = \varepsilon.
	\end{align*}
	As a consequence, we have indeed that $A\subset B$, which concludes the proof.
\end{proof}
\begin{proof}[Proof of Theorem \ref{thm:mppSpace_is_CSMS}.(ii)]
	\textbf{Step 1.} We first show that $\Phi_A$ is $\boundedlyFiniteMeasuresBorel{\anyCSMS}$-measurable for all bounded closed set $A$. Let $n\in\integers$. We prove that $I:=\{\xi\in\boundedlyFiniteMeasures{\anyCSMS}\,:\xi(A)\leq n\}$ is an open set of $\boundedlyFiniteMeasures{\anyCSMS}$, implying that $\Phi_A$ is indeed $\boundedlyFiniteMeasuresBorel{\anyCSMS}$-measurable. 
	If $A=\emptyset$, then $I=\boundedlyFiniteMeasures{\anyCSMS}$, which is open. From now on, we assume that $A\neq\emptyset$.
	Let $\mu\in I$ ($I$ is clearly not empty).
	Let $\delta\in(0,1)$ such that $\mu(A)=\mu(A^{\delta})$ (this is always possible since $\mu$ has a finite number of atoms in $A^{\gamma}\setminus A$, with $\gamma=1$, say). Let $R>0$ such that $A^{\delta}\subset B_{R}$. Consider the open ball $J:=\{\nu\in\boundedlyFiniteMeasures{\anyCSMS}\,:\,\weakHashDistance{\mu}{\nu}<\varepsilon\}$ with $\varepsilon=e^{-R}\delta/(1+\delta)$. We then have that $J\subset I$, which implies that $I$ is open. Indeed, let $\nu\in J$ and, by contradiction, assume that $\nu(A)>\mu(A^{\delta}) + \delta$. Then, for all $r\geq R$, this implies that
	\begin{equation*}
		\nu^{(r)}(A) = \nu(A) > \mu(A^{\delta}) + \delta = \mu^{(r)}(A^{\delta}) +\delta,
	\end{equation*}
	which means that $\prohorovDistance{\mu^{(r)}}{\nu^{(r)}}\geq \delta$. Hence,
	\begin{equation*}
		\weakHashDistance{\mu}{\nu}\geq \int_{R}^{\infty}e^{-r}\frac{\delta}{1+\delta}dr = \varepsilon,
	\end{equation*}
	which contradicts the assumption that $\nu\in J$. As a consequence, we must have that $\nu(A)\leq \mu(A^{\delta}) + \delta = \mu(A) +\delta$. Since, $\nu(A)\in\integers$, $\mu(A)\in\integers$ and $\delta<1$, this implies that $\nu(A)\leq\mu(A)\leq n$, and thus $\nu\in I$.
	
	\textbf{Step 2.} Consider the class $\mathcal{C}$ of sets $\mathcal{C}:=\{A\in\anyCSMSBorel\,:\, \Phi_A \mbox{ is } \boundedlyFiniteMeasuresBorel{\anyCSMS}\mbox{-measurable}\}$. By the continuity of measures \citep[Lemma 1.14 p.~8]{kallenberg2006foundations}, we have that $\Phi_{A_n} \uparrow \Phi_{A}$ for any sequence $A_{n}\uparrow A$, and since the limit of measurable functions is measurable \citep[Lemma 1.9 p.~6]{kallenberg2006foundations}, we have that $\mathcal{C}$ is closed under increasing limits. In other words, $\mathcal{C}$ forms a monotone class.
	Moreover, consider the class $\mathcal{R}$ of sets of the form $\bigcup_{i=1}^{n} A_i\setminus B_i$ where $n\in\integers$ and $A_i,B_i\in\anyCSMSBorel$ are bounded closed sets such that $(A_i\setminus B_i) \cap (A_j \setminus B_j) = \emptyset$ as soon as $i\neq j$ (i.e., we consider finite disjoint unions of differences of bounded closed sets). One can check that $\mathcal{R}$ is stable by finite intersections and symmetric differences (perhaps the most difficult is to see that, for any bounded closed sets $A_1, A_2, B_1, B_2$, the difference $(A_1\setminus B_1)\setminus (A_2\setminus B_2)$ can be written as a disjoint union of differences of bounded closed sets). This means that $\mathcal{R}$ forms a ring. Besides, for any bounded closed sets $A,B\in\anyCSMS$, since $\xi(A)<\infty$ for all $\xi\in\boundedlyFiniteMeasures{\anyCSMS}$, we have that $\Phi_{A\setminus B} = \Phi_{A\setminus(A\cap B)} = \Phi_{A} - \Phi_{A\cap B}$. As $A\cap B$ is still a bounded closed set, by applying the first part of the proof, we obtain that $\Phi_{A\setminus B}$ is measurable. By the countable additivity of measures, this implies that $\Phi_A$ is measurable for any set $A\in\mathcal{R}$, and thus $\mathcal{R}\subset \mathcal{C}$. By the monotone class theorem \citep[p.~369]{daleyVereJonesVolume1}, we then have that $\sigma(\mathcal{R})\subset \mathcal{C}$. But $\mathcal{R}$ contains all the bounded closed balls and any open set in $\anyCSMS$ is a countable union of those since $\anyCSMS$ is separable. As a consequence, we must have that $\anyCSMSBorel = \sigma(\mathcal{R})\subset \mathcal{C}$, meaning that $\Phi_A$ is measurable for all $A\in\anyCSMSBorel$.
		
	\textbf{Step 3.} To show that $\boundedlyFiniteMeasuresBorel{\anyCSMS}$ is actually generated by all mappings $\Phi_A$, $A\in\anyCSMSBorel$, consider any $\sigma$-algebra $\mathcal{R}$ on $\boundedlyFiniteMeasures{\anyCSMS}$ such that all mappings $\Phi_A$ are measurable. Then, all the sets of the form $\eqref{eq:basis_weak_hash_topology}$ should belong to $\mathcal{R}$ and, by Proposition \ref{prop:basis_weak_hash_topology}, these sets form a basis for the $\weakHash$-topology. Since $\boundedlyFiniteMeasures{\anyCSMS}$ is separable, any open set of the $\weakHash$-topology can be represented as a countable union of these sets and, thus, $\boundedlyFiniteMeasuresBorel{\anyCSMS}\subset \mathcal{R}$.
\end{proof}

%%%%%%%%%%%%%%%%%%%%%%%%%%%%%%%%%%%%%%%%%%%%%%%%%%%%%
\section*{Acknowledgements}
%%%%%%%%%%%%%%%%%%%%%%%%%%%%%%%%%%%%%%%%%%%%%%%%%%%%%
I would like to thank Mikko S. Pakkanen, Daryl Daley, Olav Kallenberg and Nicholas Bingham for helpful discussions. I gratefully acknowledge the Mini-DTC scholarship awarded by the Mathematics Department of Imperial College London.

%%%%%%%%%%%%%%%%%%%%%%%%%%%%%%%%%%%%%%%%%%%%%%%%%%%%%
\bibliography{PhD_thesis_bibli.bib}

\begin{thebibliography}{}

\bibitem[\protect\astroncite{Daley and
  Vere-Jones}{2003}]{daleyVereJonesVolume1}
Daley, D.~J. and Vere-Jones, D. (2003).
\newblock {\em An Introduction to the Theory of Point Processes. {V}ol. {I}}.
\newblock Springer, New York, second edition.

\bibitem[\protect\astroncite{Daley and
  Vere-Jones}{2008}]{daleyVereJonesVolume2}
Daley, D.~J. and Vere-Jones, D. (2008).
\newblock {\em An Introduction to the Theory of Point Processes. {V}ol. {II}}.
\newblock Springer, New York, second edition.

\bibitem[\protect\astroncite{Kallenberg}{2002}]{kallenberg2006foundations}
Kallenberg, O. (2002).
\newblock {\em Foundations of Modern Probability}.
\newblock Springer, New York, second edition.

\bibitem[\protect\astroncite{Kallenberg}{2017}]{kallenberg2017random}
Kallenberg, O. (2017).
\newblock {\em Random Measures, Theory and Applications}.
\newblock Springer, Cham.

\bibitem[\protect\astroncite{Matthes
  et~al.}{1974}]{matthes:1974:inifinitelyDivisiblePPs}
Matthes, K., Kerstan, J., and Mecke, J. (1974).
\newblock {\em Infinitely {D}ivisible {P}oint {P}rocesses}.
\newblock Akademie-Verlag, Berlin.

\bibitem[\protect\astroncite{{Morariu-Patrichi} and
  {Pakkanen}}{2017}]{morariu:2017:hybrid}
{Morariu-Patrichi}, M. and {Pakkanen}, M.~S. (2017).
\newblock {Hybrid marked point processes: characterisation, existence and
  uniqueness}.
\newblock Preprint, available at: \url{http://arxiv.org/abs/1707.06970}.

\end{thebibliography}
\bibliographystyle{apa}

\end{document}